\documentclass[11pt]{amsart}
\usepackage[margin=1in]{geometry}
\usepackage{latexsym}
\usepackage{amsfonts}
\usepackage{amsmath}
\usepackage{amssymb}
\usepackage{amsthm}
\usepackage{enumerate}
\usepackage[hang,flushmargin]{footmisc}
\usepackage{caption}
\usepackage{tabu}
\usepackage{mathrsfs}
\usepackage{subfig}
\usepackage[foot]{amsaddr}

\usepackage{graphicx}
\usepackage{epstopdf}
\usepackage{epsfig}
\usepackage{caption}
 
\usepackage{bm} 
\usepackage{cite}

\DeclareMathOperator{\VHC}{VHC}

\newtheorem{theorem}{Theorem}[section]
\newtheorem{lemma}[theorem]{Lemma}
\newtheorem{proposition}[theorem]{Proposition}
\newtheorem{corollary}[theorem]{Corollary}

\newtheorem{conjecture}[theorem]{Conjecture}

\theoremstyle{definition}
\newtheorem{definition}[theorem]{Definition}

\newtheorem{example}[theorem]{Example}

\usepackage{tikz}
\usetikzlibrary{positioning}

\makeatletter
\def\absdot{\@ifnextchar[{\@absdotlabel}{\@absdotnolabel}}
	\def\@absdotlabel[#1]#2{%
		\node at #2 {\normalsize \mybullet};
		\node at #2 [below=2pt] {\ensuremath{#1}};
	}
	\def\@absdotnolabel#1{%
		\node at #1 {\normalsize \mybullet};
	}

\def\abssquare{\@ifnextchar[{\@abssquarelabel}{\@abssquarenolabel}}
	\def\@abssquarelabel[#1]#2{%
		\node at #2 {\normalsize \mysquare};
		\node at #2 [below=2pt] {\ensuremath{#1}};
	}
	\def\@abssquarenolabel#1{%
		\node at #1 {\normalsize \mysquare};
	}

\def\abstriangle{\@ifnextchar[{\@abstrianglelabel}{\@abstrianglenolabel}}
	\def\@abstrianglelabel[#1]#2{%
		\node at #2 {\normalsize \mytriangle};
		\node at #2 [below=2pt] {\ensuremath{#1}};
	}
	\def\@abstrianglenolabel#1{%
		\node at #1 {\normalsize \mytriangle};
	}

\newcommand{\arc}[2]{
	\draw[thick] (#1,0) arc (180:0:{(#2-#1)/2});
	\node at (#1,0) {\normalsize \ensuremath{\bullet}};
	\node at (#2,0) {\normalsize \ensuremath{\bullet}};
}

\newcommand{\matching}[2][0]{
	\foreach \i/\j in {#2} {
		\arc{\i + #1}{\j + #1};
	};
}

\makeatother
\newcommand\mybullet{\raisebox{-5pt}{\scalebox{1.5}{\normalsize \ensuremath{\bullet}}}}
\newcommand\mycoloredbullet[1]{\raisebox{-5pt}{\scalebox{1.5}{\normalsize \ensuremath{\color{#1} \bullet}}}}

\newcommand{\absdotcolorlabel}[3]{%
	\node at (#1,#2) {\normalsize \mycoloredbullet{#3}};
	\node at (-0.25,#2) {\color{#3} #2};
}
\newcommand\mycoloredstar[1]{\raisebox{-5pt}{\normalsize \ensuremath{\color{#1} \bigstar}}}

\newcommand{\absstarcolorlabel}[3]{%
	\node at (#1,#2) {\normalsize \mycoloredstar{#3}};
	\node at (-0.25,#2) {\color{#3} #2};
}
\newcommand\mysquare{\raisebox{-5pt}{\footnotesize \ensuremath{\blacksquare}}}
\newcommand\mycoloredsquare[1]{\raisebox{-5pt}{\normalsize \ensuremath{\color{#1} \blacksquare}}}

\newcommand{\abssquarecolorlabel}[3]{%
	\node at (#1,#2) {\normalsize \mycoloredsquare{#3}};
	\node at (-0.25,#2) {\color{#3} #2};
}
\newcommand\mytriangle{\scalebox{1.4}{\raisebox{-5pt}{\normalsize \ensuremath{\blacktriangle}}}}
\newcommand\mycoloredtriangle[1]{\scalebox{1.4}{\raisebox{-5pt}{\normalsize \ensuremath{\color{#1} \blacktriangle}}}}

\newcommand{\abstrianglecolorlabel}[3]{%
	\node at (#1,#2) {\normalsize \mycoloredtriangle{#3}};
	\node at (-0.25,#2) {\color{#3} #2};
}
\newcommand{\plotperm}[1]{%
	\foreach \j [count=\i] in {#1} {
		\absdot{(\i,\j)};
	};
}

\newcommand{\hook}[5][black]{%
	\draw [draw=#1, very thick] (#2,#3) -- (#2,#5) -- (#4,#5);
}

\newcommand{\gap}{\kern -0.15em}

\usepackage{color}
\definecolor{lightgray}{rgb}{0.8, 0.8, 0.8}
\definecolor{darkgray}{rgb}{0.7, 0.7, 0.7}
\definecolor{darkgreen}{rgb}{0,0.5,0}
\begin{document}
\title{Stack-Sorting, Set Partitions, and Lassalle's Sequence}

\author{Colin Defant$^1$}
\address{$^1$Princeton University}
\email{cdefant@princeton.edu}

\author{Michael Engen$^2$}
\address{$^2$University of Florida}
\email{engenmt@ufl.edu}

\author{Jordan A. Miller$^3$}
\address{$^3$Washington State University}
\email{jordan.a.miller@wsu.edu}

\begin{abstract}
We exhibit a bijection between recently-introduced combinatorial objects known as valid hook configurations and certain weighted set partitions. When restricting our attention to set partitions that are matchings, we obtain three new combinatorial interpretations of Lassalle's sequence. One of these interpretations involves permutations that have exactly one preimage under the (West) stack-sorting map. We prove that the sequences obtained by counting these permutations according to their first entries are symmetric, and we conjecture that they are log-concave. We also obtain new recurrence relations involving Lassalle's sequence and the sequence that enumerates valid hook configurations. We end with several suggestions for future work.  
\end{abstract}

\maketitle

\bigskip

\section{Introduction}

In 2012, Lassalle \cite{Lassalle} introduced a sequence $(A_m)_{m\geq 1}$ defined by the recurrence relation 
\[
	A_m
	=
	(-1)^{m-1}C_m+\sum_{j=1}^{m-1}(-1)^{j-1}{2m-1\choose 2m-2j-1}A_{m-j}C_j
\]
and subject to the initial condition $A_1=1$. Here, $C_n=\frac{1}{n+1}{2n\choose n}$ is the $n^{\text{th}}$ Catalan number. The first few terms of this sequence, which has now come to be known as \emph{Lassalle's sequence}, are 
\[
	1, 1, 5, 56, 1092, 32670, 1387815, 79389310, 5882844968, 548129834616.
\] 
It is not at all obvious from the definition that the terms of Lassalle's sequence should be positive. Indeed, Lassalle's primary focus was to prove that the terms are positive and increasing, settling a conjecture of Zeilberger. This was reproven in \cite{Amdeberhan}, and the sequence was studied further in \cite{Chen, Josuat, Tevlin, Wang}. In particular, Josuat-Verg\`es found a combinatorial interpretation of $A_m$ in terms of certain weighted matchings; we briefly discuss this in Section 2. In a private communication with Lassalle, Novak pointed out that the numbers $(-1)^{m-1}A_m$ are the classical cumulants of the standard semicircular law. 

One of the primary purposes of this article is to provide three new combinatorial interpretations of the numbers $A_m$. The equivalence of these three interpretations follows from known results, but it is useful to have a variety of perspectives. The first interpretation answers a very natural question concerning the (West) stack-sorting map, whose background we now review. 

Throughout this article, the word \emph{permutation} refers to a permutation of a finite set of positive integers. We write permutations as words. Let $S_n$ denote the set of permutations of $\{1,\ldots,n\}$. A permutation is called \emph{normalized} if it is an element of $S_n$ for some $n$. 

In his seminal monograph \emph{The Art of Computer Programming}, Knuth introduced an algorithm that ``sorts" permutations through the use of a vertical ``stack" \cite{Knuth}. West later studied a slight variant of this algorithm in his 1990 Ph.D. thesis \cite{West}. Specifically, West studied the function $s$, known as the \emph{stack-sorting map}, defined by the following procedure. Suppose we are given an input permutation $\pi=\pi_1\cdots\pi_n$. At any point in time during the procedure, if the next entry in the input permutation is smaller than the entry at the top of the stack or if the stack is empty, the next entry in the input permutation is placed at the top of the stack. Otherwise, the entry at the top of the stack is annexed to the end of the growing output permutation. This algorithm terminates when the output permutation has length $n$, and $s(\pi)$ is defined to be this output permutation. The following figure illustrates this procedure and shows that $s(3142)=1324$.  

\begin{center}
\includegraphics[width=1\linewidth]{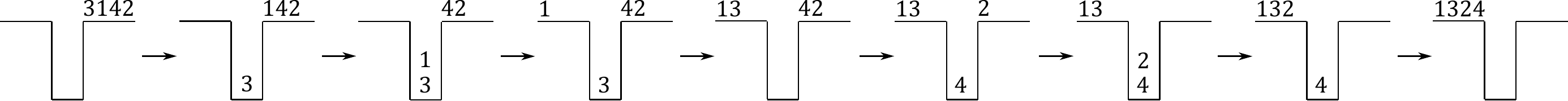}
\end{center}

We will not attempt to discuss all of the literature concerning the stack-sorting map. Instead, we state only some of the background information and refer the interested reader to \cite{Bona, BonaSurvey, Defant2}. 

West called $|s^{-1}(\pi)|$ the \emph{fertility} of the permutation $\pi$ and computed the fertilities of a few very special types of permutations \cite{West}. Bousquet-M\'elou later studied so-called \emph{sorted} permutations, which are permutations that have positive fertilities \cite{Bousquet}. We say a permutation is \emph{uniquely sorted} if its fertility is $1$. A \emph{descent} of a permutation $\pi=\pi_1\cdots\pi_n$ is an index $i\in\{1,\ldots,n-1\}$ such that $\pi_i>\pi_{i+1}$. Suppose $\pi\in S_n$ has exactly $k$ descents. We show in Section 3 that $\pi$ is uniquely sorted if and only if it is sorted and $n=2k+1$. In particular, there are no uniquely sorted permutations of even length.
When considering uniquely sorted permutations of odd length, we are led to our first combinatorial interpretation of Lassalle's sequence. Specifically, we will show that $A_{k+1}$ is precisely the number of uniquely sorted permutations in $S_{2k+1}$. 

One of the central notions concerning the stack-sorting map is that of a \emph{$t$-stack-sortable permutation}. This is simply a permutation $\pi\in S_n$ such that $s^t(\pi)=123\cdots n$, where $s^t$ denotes the $t^\text{th}$ iterate of $s$. Let $W_t(n)$ denote the number of $t$-stack-sortable permutations of length $n$. It follows from Knuth's work in \cite{Knuth} that $W_1(n)=C_n$. West conjectured \cite{West}, and Zeilberger later proved \cite{Zeilberger}, that 
\[
	W_2(n)
	=
	\displaystyle\frac{2}{(n+1)(2n+1)}{3n\choose n}.
\]
It follows from a general result of Backelin, West, and Xin \cite{Backelin} that 
\[
	W_t(n)
	\leq 
	(t+1)^{2n}
\] 
(see \cite[Theorem 3.4]{BonaSurvey}).
For several years, this was the best known upper bound for $W_t(n)$ when $t\geq 3$; it is still the best known upper bound when $t\geq 5$. The first author \cite{Defant2} improved these bounds when $t=3$ and $t=4$, showing that 
\[
	W_3(n)
	<
	n^2\cdot12.53296^n
	\hspace{.5cm}\text{and}\hspace{.5cm}
	W_4(n)
	<
	n^5\cdot21.97225^n.
\]
Recently, he found a polynomial-time algorithm for computing the numbers $W_3(n)$ \cite{DefantCounting}. 
 
In her study of sorted permutations, Bousquet-M\'elou mentioned that it would be interesting to obtain a method for computing the fertility of any given permutation. This was achieved (in greater generality) in \cite{Defant} using new combinatorial objects called ``valid hook configurations." Roughly speaking, a valid hook configuration of a permutation $\pi$ is a configuration of L-shaped ``hooks" that connect points in the plot of $\pi$ subject to certain restrictions. When we speak of a valid hook configuration on $n$ points, we simply mean a valid hook configuration of some permutation of length $n$. The theory of valid hook configurations was the key ingredient used in \cite{Defant2} in order to obtain the above-mentioned upper bounds for $W_3(n)$ and $W_4(n)$. 

We lack a thorough understanding of valid hook configurations; as a consequence, several questions concerning the stack-sorting map remain out of reach. Therefore, one of the other main purposes of this paper is to study these new objects. In Section 3, we review the major definitions and results concerning valid hook configurations. Our presentation differs slightly from that given in \cite{Defant} and \cite{Defant2}. Because those two papers focus on using valid hook configurations to prove other results, they define these structures in fairly technical terms. In contrast, our approach in the current paper is meant to elucidate the constructions and aid comprehension. We also discuss valid hook configurations of uniquely sorted permutations. This allows us to obtain our second interpretation of Lassalle's sequence. Namely, $A_{k+1}$ is the number of normalized valid hook configurations on $2k+1$ points that use exactly $k$ hooks. 

Our final interpretation of Lassalle's sequence, given in Section 4, states that $A_{k+1}$ counts the number of decreasing binary plane trees with some specific properties. The advantage of viewing Lassalle's sequence in terms of trees is that we will be able to easily detect a simple recursive combinatorial construction that builds objects counted by Lassalle's sequence from smaller such objects. 

In order to prove that these objects are counted by Lassalle's sequence, we actually establish a bijection from the set of normalized valid hook configurations on $n$ points to the set $\widetilde{\mathcal P}^c(n+1)$ defined in Section 2. This is done in Section 5. The bijection provides an interesting new way of viewing valid hook configurations. We will show that when $n=2k+1$, the preimage of the set $\widetilde{\mathcal M}^c(2k+2)$ (also defined in Section 2) under this map is the set of normalized valid hook configurations on $2k+1$ points that use exactly $k$ hooks, proving our second new combinatorial interpretation of Lassalle's sequence. The first and third interpretations then follow from the second interpretation and known results. 

We also show that the sequences $(A_{k+1}(\ell))_{\ell=1}^{2k+1}$ are symmetric, where $A_{k+1}(\ell)$ denotes the number of uniquely sorted permutations in $S_{2k+1}$ with first entry $\ell$. This is interesting since we expect the stack-sorting map to output permutations that are in some sense ``close" to the identity permutation. In other words, one should expect permutations with large fertilities to start with small numbers. On the other hand, one should expect permutations with low fertilities to start with large numbers. The symmetry in the sequences $(A_{k+1}(\ell))_{\ell=1}^{2k+1}$ says that a fertility of $1$ is not too big and not too small. This actually makes perfect sense because $1$ is the average fertility of a permutation. To conclude Section 4, we show that $A_{k+1}(\ell)$ also counts uniquely sorted permutations according to another statistic that we call the \emph{eye} of the permutation.  

In Section 5, we exploit the structures of valid hook configurations in order to obtain a recurrence relation that generates the numbers $-k_n(-1)=|\widetilde{\mathcal P}^c(n)|$ (defined in Section 2). It turns out that the same recurrence with different initial conditions generates the Lassalle numbers $A_n$. Finally, we end with several open problems and suggestions for future work.

\section{Lassalle's Sequence and Free Probability}
In this section, we review the combinatorial interpretation of Lassalle's sequence that Josuat-Verg\`es found. Let $\mathcal P(n)$ denote the collection of partitions of the set $\{1,\ldots,n\}$. If $\rho\in\mathcal P(n)$, we say two distinct blocks $B$ and $B'$ of $\rho$ form a \emph{crossing} if there exist $i,k\in B$ and $j,\ell\in B'$ such that either $i<j<k<\ell$ or $i>j>k>\ell$. The \emph{crossing graph} $G(\rho)$ is the graph whose vertices are the blocks of $\rho$ in which two blocks are adjacent if and only if they form a crossing. We say a partition $\rho\in\mathcal P(n)$ is \emph{connected} if $G(\rho)$ is a connected graph. Let $\mathcal P^c(n)$ denote the set of connected partitions in $\mathcal P(n)$. A \emph{matching} is a set partition in which every block has exactly $2$ elements. Let $\mathcal M(n)$ denote the set of matchings in $\mathcal P(n)$, and put $\mathcal M^c(n)=\mathcal M(n)\cap\mathcal P^c(n)$. 

In free probability theory, the free counterpart of the classical Poisson law is known as the \emph{free Poisson law}. It is characterized by the fact that all of the free cumulants are equal to a single parameter $\lambda>0$, known as the \emph{rate}. The free Poisson law also appears in random matrix theory in relation to Wishart matrices \cite{Nica}. The $n^{\text{th}}$ moment is given by 
\[
	m_n(\lambda)
	=
	\sum_{k=1}^n\lambda^kN(n,k),
\]
where $N(n,k)=\frac{1}{n}{n\choose k}{n\choose k-1}$ is a Narayana number. Define the classical cumulants $k_n(\lambda)$ of the free Poisson law by 
\[
	\sum_{n\geq 1}k_n(\lambda)\frac{z^n}{n!}
	=
	\log\left(1+\sum_{n\geq 1}m_n(\lambda)\frac{z^n}{n!}\right).
\]

Let $T_G(x,y)$ denote the Tutte polynomial of a finite simple graph $G$ (see \cite{Josuat} for the definition of the Tutte polynomial of a graph). Josuat-Verg\`es has proven that 
\begin{equation}
\label{Eq1}
	k_n(\lambda)
	=
	-\sum_{\rho\in\mathcal P^c(n)}(-\lambda)^{\#\rho}T_{G(\rho)}(1,0),
\end{equation}
where $\#\rho$ is the number of blocks of $\rho$ (see \cite[Theorem 7.1]{Josuat}). A \emph{source} in a directed graph is a vertex with in-degree $0$. We will make use of the following theorem due to Greene and Zaslavsky. 

\begin{theorem}[\!\!\cite{Greene}]
\label{Thm3}
Fix a vertex $v$ in a simple graph $G$. The number of acyclic orientations of $G$ in which $v$ is the unique source is $T_G(1,0)$.  
\end{theorem} 

We can apply this theorem to the crossing graph of a partition $\rho\in\mathcal P^c(n)$ to see that $T_{G(\rho)}(1,0)$ is the number of acyclic orientations of $G(\rho)$ such that the block of $\rho$ containing the element $n$ is the only source. Let $\widetilde{\mathcal P}^c(n)$ be the set of ordered pairs $(\rho,\alpha)$, where $\rho\in\mathcal P^c(n)$ and $\alpha$ is an acyclic orientation of $G(\rho)$ whose only source is the block containing $n$. According to \eqref{Eq1}, we have\footnote{Although the free Poisson law is usually defined for $\lambda>0$, Josuat-Verg\`es' proof does not rely on the positivity of $\lambda$.} 
\begin{equation}\label{Eq2}
	-k_n(-1)
	=
	\left\lvert\widetilde{\mathcal P}^c(n)\right\rvert.
\end{equation} 
By studying the cumulants of the $q$-semicircular law, Josuat-Verg\`es also proved that 
\begin{equation}\label{Eq3}
	A_m
	=
	\left\lvert\widetilde{\mathcal M}^c(2m)\right\rvert,
\end{equation}
where $\widetilde{\mathcal M}^c(2m)$ is the collection of ordered pairs $(\rho,\alpha)\in\widetilde{\mathcal P}^c(2m)$ such that $\rho$ is a matching.

\section{Valid Hook Configurations}

Valid hook configurations were introduced in \cite{Defant, Defant2} as a tool for computing fertilities of permutations. The reader wishing to compare our treatment with that given in \cite{Defant,Defant2} should be aware that the definition given here is, strictly speaking, different from the one given in those two papers. Specifically, the valid hook configurations in those articles were originally defined so that some hooks have horizontal length $1$. We have ignored these ``small'' hooks in our definition since they do not give any additional information relevant for our purposes. The reader who is seeing valid hook configurations here for the first time can ignore these remarks. 

Let us begin the definition by choosing a permutation $\pi=\pi_1\cdots\pi_n$ with descents $d_1<\cdots<d_k$ (we do not require $\pi$ to be normalized). Our running example will be the permutation 
\[ 
	2\,7\,3\,5\,9\,10\,11\,4\,8\,1\,6\,12\,13\,14\,15\,16.
\]
The \emph{plot} of $\pi$ is the graph displaying the points $(i,\pi_i)$ for $1\leq i\leq n$. Figure~\ref{fig-perm-plot} portrays the plot of our example permutation. 
We say a point $(i,\pi_i)$ is a \emph{descent top} if $i$ is a descent. Thus, the descent tops are precisely the points $(d_1,\pi_{d_1}),\ldots,(d_k,\pi_{d_k})$. In our example, the descent tops are $(2,7)$, $(7,11)$, and $(9,8)$. A \emph{hook} of $\pi$ is a sideways L shape that connects a point $(i,\pi_i)$ to a point $(j,\pi_j)$ such that $i<j$ and $\pi_i<\pi_j$. The points $(i,\pi_i)$ and $(j,\pi_j)$ are called the \emph{southwest endpoint} and the \emph{northeast endpoint}, respectively. 

\begin{footnotesize}
\begin{figure}
\begin{tikzpicture}[scale=0.35]
	\draw [lightgray] (0.5,0.5) grid (16.5,16.5);
	\foreach \val [count=\idx] in {2,7,3,5,9,10,11,4,8,1,6,12,13,14,15,16} {
		\absdot{(\idx,\val)}
		\node at (-0.1,\val) {$\val$};
	} 
\end{tikzpicture}
\caption{The plot of the permutation $2\,7\,3\,5\,9\,10\,11\,4\,8\,1\,6\,12\,13\,14\,15\,16$.}
\label{fig-perm-plot}
\end{figure}
\end{footnotesize}

\begin{definition}\label{Def1}
Let $\pi=\pi_1,\ldots,\pi_n$ be a permutation, and let $d_1<\cdots<d_k$ be the descents of $\pi$. A \emph{valid hook configuration} of $\pi$ is a tuple $\mathcal H=(H_1,\ldots,H_k)$ of hooks of $\pi$ subject to the following restrictions: 

\begin{enumerate}[1.]
	\item For every $i\in\{1,\ldots,k\}$, the southwest endpoint of $H_i$ is the descent top $(d_i,\pi_{d_i})$. 
	
	\item A point in the plot cannot lie directly above a hook. 
	
	\item Hooks cannot intersect each other except in the case that the northeast endpoint of one hook is the southwest endpoint of the other. 
\end{enumerate}  
\end{definition}

Figure \ref{fig-forbidden-hooks} shows four placements of hooks that are forbidden by conditions 2 and 3 in Definition~\ref{Def1}.
Figure \ref{fig-valid-hooks} shows a valid hook configuration of our example permutation.
Note that the total number of hooks in a valid hook configuration of $\pi$ is exactly $k$, the number of descents of $\pi$. We say a valid hook configuration is \emph{normalized} if it is a valid hook configuration of a normalized permutation. 

\begin{footnotesize}
\begin{figure}
\begin{tabular}{ccccccc}
	\begin{tikzpicture}[scale=0.35]
	\plotperm{1,3,2}
	\hook{1}{1}{3}{2}
	\end{tikzpicture}
	&\quad&
	\begin{tikzpicture}[scale=0.35]
	\plotperm{2,1,3,4}
	\hook{1}{2}{3}{3}
	\hook{2}{1}{4}{4}
	\end{tikzpicture}
	&\quad&
	\begin{tikzpicture}[scale=0.35]
	\absdot{(1,1)}
	\absdot{(3,2)}
	\absdot{(4,3)}
	\hook{1+0.05}{1}{3}{2}
	\hook{1-0.05}{1}{4}{3}
	\end{tikzpicture}
	&\quad&
	\begin{tikzpicture}[scale=0.35]
	\absdot{(1,2)}
	\absdot{(2,1)}
	\absdot{(3,4)}
	\hook{1}{2}{3}{4+0.05}
	\hook{2}{1}{3}{4-0.05}
	\end{tikzpicture}	
\end{tabular}
\caption{Four placements of hooks that are forbidden in a valid hook configuration.}
\label{fig-forbidden-hooks}
\end{figure}
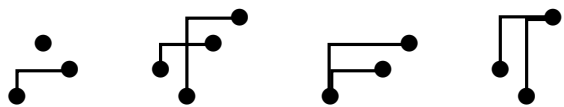
\end{footnotesize} 

\begin{footnotesize}
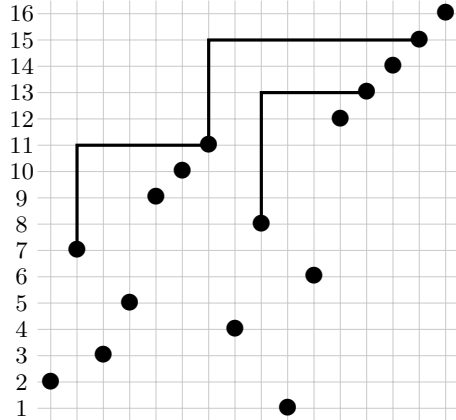
\begin{figure}
\begin{tikzpicture}[scale=0.35]
	\draw [lightgray] (0.5,0.5) grid (16.5,16.5);
	\foreach \val [count=\idx] in {2,7,3,5,9,10,11,4,8,1,6,12,13,14,15,16} {
		\absdot{(\idx,\val)}
		\node at (-0.1,\val) {$\val$};
	}
	\hook{2}{7}{7}{11}
	\hook{7}{11}{15}{15}
	\hook{9}{8}{13}{13}
\end{tikzpicture}
\caption{A valid hook configuration of $2\,7\,3\,5\,9\,10\,11\,4\,8\,1\,6\,12\,13\,14\,15\,16$.}
\label{fig-valid-hooks}
\end{figure}
\end{footnotesize}

Each valid hook configuration of $\pi$ induces a coloring of the points in the plot of $\pi$. To begin the process of coloring the plot, we first draw a ``sky" over the entire diagram. Of course, we color the sky blue. Next, assign distinct colors other than blue to the $k$ hooks in the valid hook configuration. 

There are $k$ northeast endpoints of hooks, and these points remain uncolored. However, all of the other $n-k$ points will be colored. In order to decide how to color a point $(i,\pi_i)$ that is not a northeast endpoint, imagine that this point simply looks directly upward. If this point sees a hook when looking upward, it receives the same color as the hook that it sees. If the point does not see a hook, it must see the sky, so it receives the color blue. There is one caveat here: if $(i,\pi_i)$ is the southwest endpoint of a hook, then it looks around (on the left side of) the vertical part of that hook. Figure~\ref{fig-colored-hooks} shows the coloring of the plot of our example permutation induced from the valid hook configuration from Figure~\ref{fig-valid-hooks}.
Observe that the point $(2,7)$ is colored blue because this point looks around the first (green) hook and sees the sky. Similarly, $(9,8)$ is red because this point looks around the third (brown) hook and sees the second (red) hook.  
\begin{footnotesize}
\begin{figure}

\begin{tikzpicture}[scale=0.35]
	
	\draw [lightgray] (0.5,0.5) grid (16.5,16.5);
	
	\def\firsthookcolor{blue}
	\def\secondhookcolor{darkgreen}
	\def\thirdhookcolor{red}
	\def\fourthhookcolor{brown}
	
	\hook[\thirdhookcolor]{7}{11}{15}{15}
	\hook[\secondhookcolor]{2}{7}{7}{11}
	%\hook[\firsthookcolor]{0.5}{0.5}{17-0.5}{17-0.5}
	\draw[\firsthookcolor,line width=1mm] (0.5,16.7) --(16.5,16.7);
	\hook[\fourthhookcolor]{9}{8}{13}{13}
	
	\absdotcolorlabel{ 1}{ 2}{\firsthookcolor}
	\absdotcolorlabel{ 2}{ 7}{\firsthookcolor}
	
	\absstarcolorlabel{ 3}{ 3}{\secondhookcolor}
	\absstarcolorlabel{ 4}{ 5}{\secondhookcolor}
	\absstarcolorlabel{ 5}{ 9}{\secondhookcolor}
	\absstarcolorlabel{ 6}{10}{\secondhookcolor}
	\absdotcolorlabel{ 7}{11}{black}
%	\absstarcolorlabel{ 7}{11}{\secondhookcolor}
	
	\abssquarecolorlabel{ 8}{ 4}{\thirdhookcolor}
	\abssquarecolorlabel{ 9}{ 8}{\thirdhookcolor}
	
	\abstrianglecolorlabel{10}{ 1}{\fourthhookcolor}
	\abstrianglecolorlabel{11}{ 6}{\fourthhookcolor}
	\abstrianglecolorlabel{12}{12}{\fourthhookcolor}
	\absdotcolorlabel{13}{13}{black}
%	\abstrianglecolorlabel{13}{13}{\fourthhookcolor}
	
	\abssquarecolorlabel{14}{14}{\thirdhookcolor}
	\absdotcolorlabel{15}{15}{black}
%	\abssquarecolorlabel{15}{15}{\thirdhookcolor}
	
	\absdotcolorlabel{16}{16}{\firsthookcolor}
%		
%		\foreach \val [count=\idx] in {2,7,3,5,9,10,11,4,8,1,6,12,13,14,15,16} {
%		%	\absdot{(\idx,\val)}
%			\node at (-0.25,\val) {$\val$};
%		}
	
\end{tikzpicture}
\caption{The coloring induced by the valid hook configuration in Figure \ref{fig-valid-hooks}. The colored points are represented with different shapes in order to make the diagram easier to understand in black and white.}
\label{fig-colored-hooks}
\end{figure}
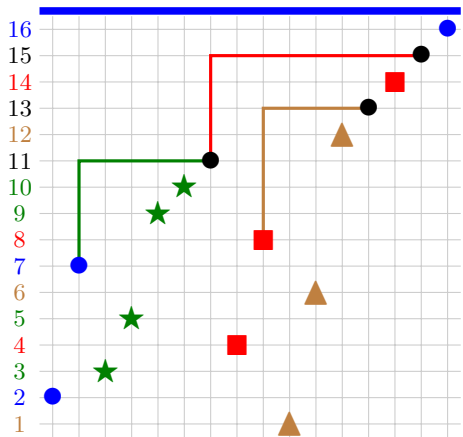
\end{footnotesize}

To summarize, we started with a permutation $\pi$ with exactly $k$ descents. We chose a valid hook configuration of $\pi$ by drawing $k$ hooks according to the rules 1, 2, and 3 in Definition \ref{Def1}. This valid hook configuration then induced a coloring of the plot of $\pi$. Specifically, $n-k$ points were colored, and $k+1$ colors were used (one for each hook and one for the sky). Let $q_i$ be the number of points colored the same color as the $i^\text{th}$ hook $H_i$, and let $q_0$ be the number of points colored blue (sky color). Then $(q_0,q_1,\ldots,q_k)$ is a composition of $n-k$ into $k+1$ parts. We call a composition obtained in this way a \emph{valid composition of }$\pi$. Let $\VHC(\pi)$ be the set of valid hook configurations of $\pi$. Let $\mathcal V(\pi)$ be the set of valid compositions of $\pi$. 

Although we will not use this fact, it is good to be aware of Lemma 3.1 from \cite{Defant2}, which states that the map $\VHC(\pi)\to\mathcal V(\pi)$ obtained by sending a valid hook configuration to its induced valid composition is a bijection. The motivation for studying valid hook configurations comes from the following theorem concerning the fertility of a permutation. 

\begin{theorem}[\!\!\cite{Defant}]
\label{Thm1}
Let $\pi$ be a permutation with exactly $k$ descents. The fertility of $\pi$ is given by the formula 
\[
	|s^{-1}(\pi)|
	=
	\sum_{(q_0,\ldots,q_k)\in\mathcal V(\pi)}\prod_{t=0}^k C_{q_t},
\]
where $C_j=\frac{1}{j+1}{2j\choose j}$ is the $j^\text{th}$ Catalan number. 
\end{theorem}

Using valid hook configurations, one can also count preimages of a permutation $\pi$ under the map $s$ according to certain statistics. For example, Corollary 5.1 in \cite{Defant} provides a formula for the number of preimages of $\pi$ with a given number of valleys. Theorem 5.2 in the same paper gives a formula for the number of preimages with a prescribed number of descents. 

One immediate application of Theorem \ref{Thm1} comes from Exercise 18 in Chapter 8 of B\'ona's \emph{Combinatorics of Permutations} \cite{Bona}, which asks for the maximum number of descents that a sorted permutation of length $n$ can have. Recall that a permutation is called \emph{sorted} if its fertility is positive. Suppose $\pi=\pi_1\cdots\pi_n$ is a sorted permutation with $k$ descents. It follows from Theorem~\ref{Thm1} that $\mathcal V(\pi)$ is nonempty. Since the elements of $\mathcal V(\pi)$ are compositions of $n-k$ into $k+1$ parts, we must have $k+1\leq n-k$. Thus, $k\leq \left\lfloor\frac{n-1}{2}\right\rfloor$. Using valid hook configurations, it is not difficult to construct sorted permutations of length $n$ with $\left\lfloor\frac{n-1}{2}\right\rfloor$ descents. If $n=2k+1$, then Theorem \ref{Thm1} actually tells us that every sorted permutation of length $n$ with $k$ descents has fertility $1$ (i.e., it is uniquely sorted). Indeed, the only valid composition of such a permutation is $(1,1,\ldots,1)$, so the fertility is $\prod_{t=0}^k C_1=1$. 

On the other hand, suppose $\pi$ is a uniquely sorted permutation of length $n$ with $k$ descents. According to Theorem \ref{Thm1}, we must have $\mathcal V(\pi)=\{(1,1,\ldots,1)\}$, where $(1,1,\ldots,1)$ is a composition of $n-k$ into $k+1$ parts that are all equal to $1$. This proves the following proposition. 

\begin{proposition}\label{Prop1}
Let $\pi$ be a permutation of length $n$ with $k$ descents. The permutation $\pi$ is uniquely sorted if and only if it is sorted and $n=2k+1$.
\end{proposition}

As mentioned in the introduction, this proves that there are no uniquely sorted permutations of even length. 

\begin{corollary}\label{Cor1}
Uniquely sorted permutations in $S_{2k+1}$ are in bijection with normalized valid hook configurations on $2k+1$ points with $k$ hooks.  
\end{corollary}

\begin{proof}
A uniquely sorted permutation in $S_{2k+1}$ has a unique valid hook configuration, which must have $k$ hooks. On the other hand, if we are given a normalized valid hook configuration on $2k+1$ points with $k$ hooks, then the underlying permutation must be a sorted permutation in $S_{2k+1}$ with $k$ descents. By Proposition \ref{Prop1}, this permutation is uniquely sorted. 
\end{proof}

Corollary \ref{Cor1} establishes the equivalence of our first two combinatorial interpretations of Lassalle's sequence. We can describe uniquely sorted permutations (or equivalently, their valid hook configurations) via the following recursive combinatorial construction. 

Begin by choosing two uniquely sorted permutations $\tau$ and $\mu$ such that $\tau\mu\in S_{2k}$ for some $k$. Make sure that the largest entry of $\tau$ is greater than the first entry of $\mu$. Now form the permutation $\tau\mu(2k+1)$. It might be easier to visualize this procedure by picturing valid hook configurations. Figure \ref{fig-hook-decomp} shows a uniquely sorted permutation with its valid hook configuration; the two permutations from which this larger permutation was formed are shaded separately. 
In general, if we are given a uniquely sorted permutation $\pi=\pi_1\cdots\pi_{2k+1}\in S_{2k+1}$, it is easy to reobtain the two uniquely sorted permutations from which it was built. We first draw the unique valid hook configuration of the given permutation. The point $(2k+1,2k+1)$ must be a point in the plot, and it must be a northeast endpoint of a hook. The southwest endpoint of that hook is of the form $(d_r,\pi_{d_r})$ for some descent $d_r$. The two permutations from which $\pi$ was built are $\tau=\pi_1\cdots\pi_{d_r}$ and $\mu=\pi_{d_r+1}\cdots\pi_{2k}$.

\begin{footnotesize}
\begin{figure}[h]

\begin{tikzpicture}[scale=0.35]

\draw[opacity=0.5, fill=lightgray, draw=darkgray, thick] (0.5,3.5) rectangle (3.4,7.5);
\draw[opacity=0.5, fill=lightgray, draw=darkgray, thick] (3.6,0.5) rectangle (8.5,8.5);

\draw [lightgray] (0.5,0.5) grid (9.5,9.5);

\foreach \val [count=\idx] in {5,4,7,6,2,1,3,8,9} {
	\absdot{(\idx,\val)}
	\node at (-0.1,\val) {$\val$};
}
\hook{1}{5}{3}{7}
\hook{3}{7}{9}{9}
\hook{4}{6}{8}{8}
\hook{5}{2}{7}{3}
\end{tikzpicture}

\caption{The uniquely sorted permutation $547621389$ is built from the smaller uniquely sorted permutations $547$ and $62138$.}
\label{fig-hook-decomp}
\end{figure}
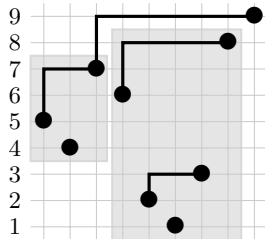
\end{footnotesize}

This recursive construction suggests a link with trees, which leads us to our third combinatorial interpretation of Lassalle's sequence. 

\section{Decreasing Plane Trees}

A decreasing plane tree is a rooted plane tree whose nodes are labeled with distinct positive integers such that every non-root node has a label that is smaller than the label of its parent. A rooted plane tree is called \emph{binary} if each vertex has at most two children. If a vertex has exactly one child, we distinguish between whether this child is a left or right child. A labeled tree is called \emph{normalized} if its set of labels is of the form $\{1,\ldots,n\}$ for some $n$. See Figure~\ref{fig-decreasing-binary-tree} for an example of a decreasing binary plane tree. 

To read a decreasing binary plane tree in \emph{in-order} (sometimes called \emph{symmetric order}), we first read the left subtree of the root, then the root, and finally the right subtree of the root. Each subtree is itself read in in-order. The in-order reading of the tree in Figure~\ref{fig-decreasing-binary-tree} is $2635741$. Let $I(T)$ denote the in-order reading of the decreasing binary plane tree $T$. The map $I$ is a bijection from the set of normalized decreasing binary plane trees on $n$ vertices to the set $S_n$ \cite{Bona, Stanley}. 

\begin{footnotesize}
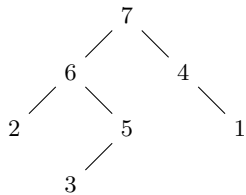
\begin{figure}[h]
\begin{tikzpicture}[scale=0.75]

\node (7) at ( 0, 0) {$7$};
\node (4) at ( 1,-1) {$4$};
\node (1) at ( 2,-2) {$1$};
\node (6) at (-1,-1) {$6$};
\node (5) at ( 0,-2) {$5$};
\node (2) at (-2,-2) {$2$};
\node (3) at (-1,-3) {$3$};

\draw (1) -- (4) -- (7) -- (6) -- (5) -- (3);
\draw (6) -- (2);

\end{tikzpicture}
\caption{A (normalized) decreasing binary plane tree.}
\label{fig-decreasing-binary-tree}
\end{figure}
\end{footnotesize}

To read a decreasing binary plane tree in \emph{postorder}, we first read the left subtree of the root, then the right subtree of the root, and finally the root. Each subtree is itself read in postorder. The postorder reading of the tree in Figure~\ref{fig-decreasing-binary-tree} is $2356147$. Let $P(T)$ denote the postorder reading of a decreasing binary plane tree $T$. It turns out that the stack-sorting map can be described in terms of in-order and postorder readings \cite{Bona}. Specifically, \[s=P\circ I^{-1}.\] For example, $s(2635741)=2356147=P(I^{-1}(2635741))$, where $I^{-1}(2635741)$ is the tree in Figure~\ref{fig-decreasing-binary-tree}. 

\begin{footnotesize}
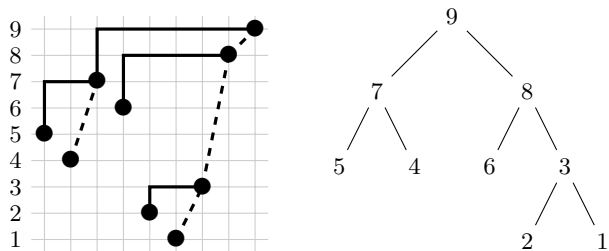
\begin{figure}
\begin{tabular}{ccc}
\begin{tikzpicture}[scale=0.35]
	\draw [lightgray] (0.5,0.5) grid (9.5,9.5);

	\foreach \val [count=\idx] in {5,4,7,6,2,1,3,8,9} {
		\absdot{(\idx,\val)}
		\node at (-0.1,\val) {$\val$};
	}
	
	\draw [very thick, dashed] (2,4) -- (3,7);
	\draw [very thick, dashed] (6,1) -- (7,3);
	\draw [very thick, dashed] (7,3) -- (8,8);
	\draw [very thick, dashed] (8,8) -- (9,9);
	
	\hook{1}{5}{3}{7}
	\hook{3}{7}{9}{9}
	\hook{4}{6}{8}{8}
	\hook{5}{2}{7}{3}
\end{tikzpicture}
&\quad
\begin{tikzpicture}

	\node (9) at ( 0, 0) {$9$};
	\node (7) at (-1,-1) {$7$};
	\node (5) at (-1.5,-2) {$5$};
	\node (4) at (-0.5,-2) {$4$};
	\node (8) at ( 1,-1) {$8$};
	\node (6) at (0.5,-2) {$6$};
	\node (3) at (1.5,-2) {$3$};
	\node (2) at ( 1,-3) {$2$};
	\node (1) at ( 2,-3) {$1$};
	
	\draw (5) -- (7) -- (4);
	\draw (7) -- (9) -- (8) -- (3) -- (2);
	\draw (8) -- (6);
	\draw (3) -- (1);
\end{tikzpicture}
\end{tabular}
\caption{The left image shows a valid hook configuration of $\pi=547621389$ along with some additional dotted lines. The hooks and dotted lines transform into the edges of the tree on the left, which is the unique decreasing binary plane tree with postorder $\pi$.}
\label{fig-construction}
\end{figure}
\end{footnotesize}
The first author has shown \cite{Defant} how to construct the decreasing binary plane trees whose postorders are equal to a given permutation $\pi$. We will not review this construction here. Instead, we simply discuss the mechanics of this construction in the (much simpler) specific case in which $\pi$ is uniquely sorted. Refer to Figure~\ref{fig-construction} for an illustration. 

Suppose $\pi=\pi_1\cdots\pi_{2k+1}$ is uniquely sorted, and draw its valid hook configuration. For each northeast endpoint $(j,\pi_j)$, consider the point $(j-1,\pi_{j-1})$. Draw a dotted line between these two points. Now replace each point $(i,\pi_i)$ with a vertex with label $\pi_i$. ``Unbend" each hook to transform it into a left edge. Similarly, transform each dotted line into a right edge. This produces the unique decreasing binary plane tree $T$ with postorder $\pi$. Note also that $I(T)$ is the unique permutation in $s^{-1}(\pi)$. 

For convenience, we say a decreasing binary plane tree $T$ is \emph{lonely} if no other decreasing binary plane tree has the same postorder as $T$. The normalized lonely trees on $n$ vertices are precisely the decreasing binary plane trees whose postorders are uniquely sorted permutations in $S_n$. This provides our third combinatorial interpretation of Lassalle's sequence. More precisely, there are no normalized lonely trees with an even number of vertices, while there are precisely $A_{k+1}$ normalized lonely trees on $2k+1$ vertices. 

It is possible to describe lonely trees without referring to permutations or the stack-sorting map. The description is recursive and is essentially equivalent to the recursive construction of uniquely sorted permutations discussed at the end of the previous section. The proof that the construction has the desired properties amounts to combining the recursive construction of uniquely sorted permutations with the above bijection between uniquely sorted permutations and normalized lonely trees.      

Given a decreasing binary plane tree $T$, we call the vertex that is read first in $I(T)$ the \emph{leftmost vertex} of $T$. Suppose $a$ is a vertex in $T$ with two children. Let $b$ be the left child of $a$. By the \emph{leftmost cousin} of $b$, we mean the leftmost vertex in the right subtree of $a$. This is also the vertex that is read immediately after $a$ in $I(T)$. The leftmost vertex of the tree in Figure~\ref{fig-decreasing-binary-tree} has label $2$. In that tree, the leftmost cousin of the vertex labeled $2$ is the vertex labeled $3$. Bousquet-M\'elou \cite{Bousquet} defined a decreasing binary plane tree to be \emph{canonical} if every vertex that has a left child also has a right child and every left child has a label that is larger than the label of its leftmost cousin. We say a decreasing binary plane tree is \emph{full} if every vertex has either $0$ or $2$ children. 

Our alternative descriptions of lonely trees are as follows. Of course, a single vertex with a positive integer label is a lonely tree. A lonely tree with more than one vertex is a decreasing binary plane tree that consists of a root whose left and right subtrees are themselves lonely and that has the additional property that the left child of the root has a label that is larger than the label of its leftmost cousin. Alternatively, a lonely tree is simply a decreasing binary plane tree that is full and canonical. 

\section{The Main Bijection}

Now that we have described our three combinatorial interpretations of Lassalle's sequence and shown that they are in bijection with each other, we can move on to actually proving that these objects are counted by Lassalle's sequence. This will follow as a consequence of the following more general theorem. First, we need a short lemma and some observations about valid hook configurations. 

\begin{lemma}\label{Lem1}
Let $\mathcal H$ be a valid hook configuration of a permutation $\pi$. Consider the coloring of the plot of $\pi$ induced by $\mathcal H$. If $i_1<\cdots<i_r$ are indices such that $(i_1,\pi_{i_1}),\ldots,(i_r,\pi_{i_r})$ are all given the same color, then $\pi_{i_1}<\cdots<\pi_{i_r}$.  
\end{lemma}

\begin{proof}
It suffices to prove the lemma in the case $r=2$. Assume instead that $\pi_{i_1}>\pi_{i_2}$. There must be a descent $d$ of $\pi$ such that $i_1\leq d<i_2$ and $\pi_d>\pi_{i_2}$. Assume that $d$ is chosen maximally subject to these conditions. There must be a hook whose southwest endpoint is $(d,\pi_d)$. The point $(i_2,\pi_{i_2})$ lies below this hook while $(i_1,\pi_{i_1})$ does not. This means that $(i_1,\pi_{i_1})$ and $(i_2,\pi_{i_2})$ cannot have the same color, which contradicts our hypothesis. 
\end{proof}

Suppose $\mathcal H$ is a valid hook configuration of a permutation $\pi$. There is a canonical decomposition of $\mathcal H$ that makes use of what we call the \emph{top hook}. This is simply the hook whose northeast endpoint is farthest to the north. For example, the top hook in Figure~\ref{fig-valid-hooks} is the hook with southwest endpoint $(7,11)$ and northeast endpoint $(15,15)$. The top hook separates $\mathcal H$ into two smaller valid hook configurations. We call these the sheltered and unsheltered pieces of $\mathcal H$. Specifically, the sheltered piece consists of all of the points and hooks that lie strictly underneath the top hook. The unsheltered piece consists of all of the other points and hooks except for the northeast endpoint of the top hook and the top hook itself. Let $\mathcal H_S$ denote the set of entries $\pi_i$ such that $(i,\pi_i)$ is in the sheltered piece. Define $\mathcal H_U$ similarly for the unsheltered piece. In the example depicted in Figure~\ref{fig-valid-hooks}, we have 
\[
	\mathcal H_S
	=
	\{1,4,6,8,12,13,14\}
	\quad\text{and}\quad
	\mathcal H_U
	=
	\{2,3,5,7,9,10,11,16\}.
\] This decomposition of valid hook configurations will be crucial in our proof of Theorem \ref{Thm2} below.

We are now ready to define our main bijection. Recall from Section 2 that $\widetilde{\mathcal P}^c(n)$ is the set of ordered pairs $(\rho,\alpha)$, where $\rho$ is a set partition of $\{1,\ldots,n\}$ whose crossing graph $G(\rho)$ is connected and $\alpha$ is an acyclic orientation of $G(\rho)$ whose only source is the block containing $0$. Let 
\[
	\VHC(S_{n-1})
	=
	\bigcup_{\pi\in S_{n-1}}\VHC(\pi)
\]
denote the set of all normalized valid hook configurations on $n-1$ points. 

We define a map \[\Phi:\VHC(S_{n-1})\to \widetilde{\mathcal P}^c(n)\] as follows. Let $\mathcal H$ be a valid hook configuration of a permutation $\pi=\pi_1\cdots\pi_{n-1}\in S_{n-1}$. Suppose $\mathcal H$ has $k$ hooks (equivalently, $\pi$ has $k$ descents). As discussed in Section 3, $\mathcal H$ induces a coloring of the plot of $\pi$. The $k$ northeast endpoints of hooks in $\mathcal H$ remain uncolored in this coloring. However, for the purpose of this proof, let us actually color the northeast endpoints as well. We do this by giving the northeast endpoint of a hook the same color as that hook. We now obtain a coloring of the elements of $\{1,\ldots,n-1\}$ by giving $\pi_i$ the same color as the point $(i,\pi_i)$ for each $i$. For example, $\pi_1$ must be blue (sky-colored) because the point $(1,\pi_1)$ must be blue. Let us also color the number $n$ blue. This yields a partition $\rho$ of $\{1,\ldots,n\}$ into color classes. For each block $B$ of this partition, let $\widehat B=\{i:\pi_i\in B, 1\leq i\leq n-1\}$.

We need to choose an acyclic orientation $\alpha$ of $G(\rho)$. To do this, suppose we have an edge of $G(\rho)$ with endpoints $B$ and $B'$. In other words, $B$ and $B'$ are blocks of $\rho$ that form a crossing. If $\min \widehat B<\min \widehat{B'}$, orient this edge from $B$ to $B'$. If $\min \widehat{B'}<\min \widehat B$, orient the edge from $B'$ to $B$. This defines the acyclic orientation $\alpha$, so put $\Phi(\mathcal H)=(\rho,\alpha)$. 

\begin{example}\label{Exam2}
Let $n=17$, and let $\mathcal H$ be the valid hook configuration in Figure~\ref{fig-valid-hooks}. To obtain the pair $\Phi(\mathcal H)=(\rho,\alpha)$, begin by coloring the diagram as in Figure~\ref{fig-colored-hooks}. For each hook $H$, color the northeast endpoint of $H$ the same color as $H$. This yields the diagram shown on the left in Figure~\ref{fig-complete-map}. The blocks in $\rho$ are the heights of the points in each color class, where we put the number $n=17$ in the blue block. Specifically, the blocks of $\rho$ are \[B_{\text{blue}}=\{2,7,16,17\},\quad B_{\text{green}}=\{3,5,9,10,11\},\quad B_{\text{red}}=\{4,8,14,15\},\quad B_{\text{brown}}=\{1,6,12,13\}.\] We also have \[\widehat{B}_{\text{blue}}=\{1,2,16\},\quad \widehat{B}_{\text{green}}=\{3,4,5,6,7\},\quad \widehat{B}_{\text{red}}=\{8,9,14,15\},\quad \widehat{B}_{\text{brown}}=\{10,11,12,13\}.\] 

In this example, every pair of blocks in $\rho$ forms a crossing, so $G(\rho)$ is a complete graph on $4$ vertices. In the acyclic orientation $\alpha$, depicted in the right image of Figure \ref{fig-complete-map}, we orient the edge connecting $B_{\text{blue}}$ and $B_{\text{green}}$ away from $B_{\text{blue}}$ since $\min \widehat{B}_{\text{blue}}=1<3=\min \widehat{B}_{\text{green}}$. We orient an edge from $B_{\text{red}}$ to $B_{\text{brown}}$ since $\min \widehat{B}_{\text{red}}=8<10=\min \widehat{B}_{\text{brown}}$. The other edges are oriented similarly.  
\end{example}

\begin{footnotesize}
\begin{figure}
\[
\begin{array}{ccc}
%	\begin{tikzpicture}[scale=0.35]
%		\draw [lightgray] (0.5,0.5) grid (16.5,16.5);
%		\foreach \val [count=\idx] in {2,7,3,5,9,10,11,4,8,1,6,12,13,14,15,16} {
%			\absdot{(\idx,\val)}
%			\node at (-0.1,\val) {$\val$};
%		}
%		\hook{2}{7}{7}{11}
%		\hook{7}{11}{15}{15}
%		\hook{9}{8}{13}{13}
%	\end{tikzpicture} &
%	
	\begin{tikzpicture}[scale=0.35]
	
		\draw [lightgray] (0.5,0.5) grid (16.5,16.5);
		
		\def\firsthookcolor{blue}
		\def\secondhookcolor{darkgreen}
		\def\thirdhookcolor{red}
		\def\fourthhookcolor{brown}
		
		\hook[\thirdhookcolor]{7}{11}{15}{15}
		\hook[\secondhookcolor]{2}{7}{7}{11}
		%\hook[\firsthookcolor]{0.5}{0.5}{17-0.5}{17-0.5}
		\draw[\firsthookcolor,line width=1mm] (0.5,16.7) --(16.5,16.7);
		\hook[\fourthhookcolor]{9}{8}{13}{13}
		
		\absdotcolorlabel{ 1}{ 2}{\firsthookcolor}
		\absdotcolorlabel{ 2}{ 7}{\firsthookcolor}
		
		\absstarcolorlabel{ 3}{ 3}{\secondhookcolor}
		\absstarcolorlabel{ 4}{ 5}{\secondhookcolor}
		\absstarcolorlabel{ 5}{ 9}{\secondhookcolor}
		\absstarcolorlabel{ 6}{10}{\secondhookcolor}
		\absstarcolorlabel{ 7}{11}{\secondhookcolor}
		
		\abssquarecolorlabel{ 8}{ 4}{\thirdhookcolor}
		\abssquarecolorlabel{ 9}{ 8}{\thirdhookcolor}
		
		\abstrianglecolorlabel{10}{ 1}{\fourthhookcolor}
		\abstrianglecolorlabel{11}{ 6}{\fourthhookcolor}
		\abstrianglecolorlabel{12}{12}{\fourthhookcolor}
		\abstrianglecolorlabel{13}{13}{\fourthhookcolor}
		
		\abssquarecolorlabel{14}{14}{\thirdhookcolor}
		\abssquarecolorlabel{15}{15}{\thirdhookcolor}
		
		\absdotcolorlabel{16}{16}{\firsthookcolor}
%		
%		\foreach \val [count=\idx] in {2,7,3,5,9,10,11,4,8,1,6,12,13,14,15,16} {
%		%	\absdot{(\idx,\val)}
%			\node at (-0.25,\val) {$\val$};
%		}
	
	\end{tikzpicture}
%	\\
	&\quad\quad\quad&
	\begin{tikzpicture}[scale=2]
	
		\def\distance{0.15}
		\def\longdistance{\distance * 0.8}
	
		\def\firsthookcolor{blue}
		\def\secondhookcolor{darkgreen}
		\def\thirdhookcolor{red}
		\def\fourthhookcolor{brown}
		
		\draw [ultra thick, ->] (0,1-\distance) to (0,0+\distance);
		\draw [ultra thick, ->] (1,1-\distance) to (1,0+\distance);
		
		\draw [ultra thick, ->] (0+\longdistance, 1-\longdistance) to (1-\longdistance,0+\longdistance);
		\draw [ultra thick, ->] (1-\longdistance,1-\longdistance) to (0+\longdistance, 0+\longdistance);
		
		\draw [ultra thick, ->] (0+\distance,0) to (1-\distance,0);
		\draw [ultra thick, ->] (0+\distance,1) to (1-\distance,1);
		
		\node at (-.1,1.21) {\color{\firsthookcolor}$\{2,7,16,17\}$};
		\node at (1.1,1.21) {\color{\secondhookcolor}$\{3,5,9,10,11\}$};
		\node at (-.1,-.22) {\color{\thirdhookcolor}$\{4,8,14,15\}$};
		\node at (1.1,-.22) {\color{\fourthhookcolor}$\{1,6,12,13\}$};
		\node at (0,1) {\scalebox{1.5}{\mycoloredbullet{\firsthookcolor}}};
		\node at (1,1) {\scalebox{1.5}{\mycoloredstar{\secondhookcolor}}};
		\node at (0,0) {\scalebox{1.5}{\mycoloredsquare{\thirdhookcolor}}};
		\node at (1,0) {\scalebox{1.5}{\mycoloredtriangle{\fourthhookcolor}}};
%		\node [circle, fill = \firsthookcolor] at (0,1) {};
%		\node [circle, fill = \secondhookcolor] at (1,1) {};
%		\node [circle, fill = \thirdhookcolor] at (0,0) {};
%		\node [circle, fill = \fourthhookcolor] at (1,0) {};
		
		\node at (0,-1) {};
	\end{tikzpicture}

\end{array}
\]
\caption{An illustration of Example \ref{Exam2}. The map $\Phi$ sends the valid hook configuration on the left to the set partition and acyclic orientation illustrated on the right. } 
\label{fig-complete-map}
\end{figure}
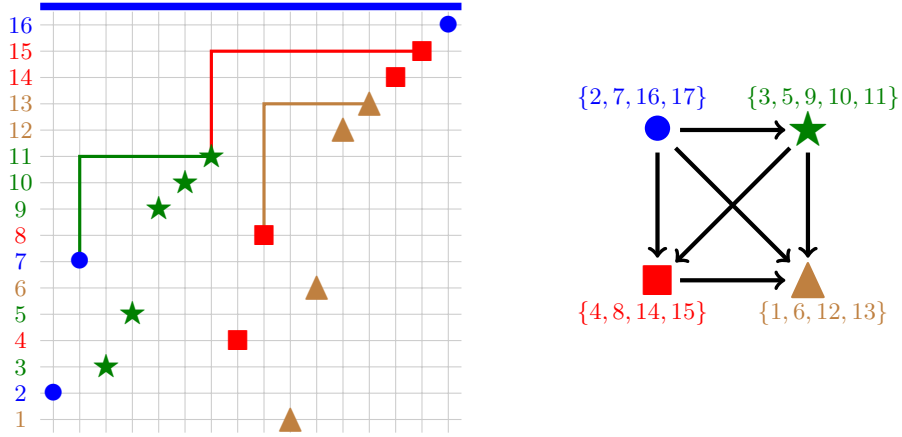
\end{footnotesize}

\begin{theorem}\label{Thm2}
The map $\Phi:\VHC(S_{n-1})\to \widetilde{\mathcal P}^c(n)$ defined above is a bijection. 
\end{theorem}

\begin{proof}
We first check that the pair $(\rho,\alpha)$ is in fact an element of $\widetilde{\mathcal P}^c(n)$. It is clear that the orientation $\alpha$ is acyclic. Let $B_{\text{blue}}$ denote the blue block of $\rho$. We know that $n\in B_{\text{blue}}$ and $\min\widehat{B}_{\text{blue}}=1$. This means that the block containing $n$ is a source; we must show that it is the only source. This will also imply that $G(\rho)$ is connected since each connected component contains a source for the acyclic orientation (this is a standard fact about acyclic orientations of graphs). 

For each color $c$ that we use, let $B_c$ be the block of $\rho$ with the color $c$. If $c$ is not blue, let $H_c$ be the hook with the color $c$. Choose a non-blue color $c_1$. Our goal is to find a block $B^*$ of $\rho$ such that $B_{c_1}$ and $B^*$ form a crossing and $\min \widehat{B}^*<\min \widehat{B}_{c_1}$. Let $(u_2,\pi_{u_2})$ be the southwest endpoint of $H_{c_1}$, and let $c_2$ be the color of $(u_2,\pi_{u_2})$. If $c_2$ is not blue, then let $(u_3,\pi_{u_3})$ be the southwest endpoint of $H_{c_2}$, and let $c_3$ be the color of $(u_3,\pi_{u_3})$. Continue in this fashion until eventually defining a point $(u_m,\pi_{u_m})$ whose color $c_m$ \emph{is} blue. It follows from the properties of valid hook configurations that $\min\widehat{B}_{c_m}<\min\widehat{B}_{c_{m-1}}<\cdots<\min\widehat{B}_{c_1}$. Therefore, it suffices to show that there is some $i\in\{2,\ldots,m\}$ such that $B_{c_1}$ and $B_{c_i}$ form a crossing (we can then put $B^*=B_{c_i}$). 

For $i\in\{2,\ldots,m\}$, the point $(u_i+1,\pi_{u_i+1})$ is the color $c_{i-1}$. This means that $\pi_{u_i+1}\in B_{c_{i-1}}$, so $\min B_{c_{i-1}}\leq \pi_{u_i+1}$. Furthermore, $\max B_{c_{i-1}}$ is the height of the northeast endpoint of $H_{c_{i-1}}$. Because $(u_i,\pi_{u_i})$ is the southwest endpoint of $H_{c_{i-1}}$, we have $\pi_{u_i}<\max B_{c_{i-1}}$. We also know that $(u_i,\pi_{u_i})$ is a descent top of $\pi$, so $\pi_{u_i+1}<\pi_{u_i}$. Combining these inequalities yields $\min B_{c_{i-1}}<\pi_{u_i}<\max B_{c_{i-1}}$. This is important because $\pi_{u_i}\in B_{c_i}$. Suppose by way of contradiction that none of the blocks $B_{c_2},\ldots,B_{c_m}$ form a crossing with $B_{c_1}$. Because $B_{c_2}$ does not form a crossing with $B_{c_1}$, we must have $\min B_{c_1}<\min B_{c_2}<\pi_{u_3}<\max B_{c_2}<\max B_{c_1}$. Because $B_{c_3}$ does not form a crossing with $B_{c_1}$, we must have $\min B_{c_1}<\min B_{c_3}<\pi_{u_4}<\max B_{c_3}<\max B_{c_1}$. Continuing in this manner, we eventually find that $\min B_{c_1}<\pi_{u_m}<\max B_{c_1}$. However, $B_{c_m}$ is the blue block, so $\pi_{u_m}$ and $n$ are in $B_{c_m}$. We have $\min B_{c_1}<\pi_{u_m}<\max B_{c_1}<n$, which means that $B_{c_m}$ does form a crossing with $B_{c_1}$ after all, a contradiction. 

It remains to show that $\Phi$ is a bijection. To do so, we exhibit its inverse. Suppose we are given a pair $(\rho,\alpha)\in\widetilde{\mathcal P}^c(n)$. We want to reobtain the valid hook configuration $\mathcal H$ with $\Phi(\mathcal H)=(\rho,\alpha)$. We can assume that $\rho$ has more than one block; otherwise, $\mathcal H$ is the valid hook configuration of the identity permutation that has no hooks. Here is where we make use of the ``top hook decomposition" discussed before the definition of the map $\Phi$. If we can determine the southwest and northeast endpoints of the top hook of $\mathcal H$ along with the sets $\mathcal H_S$ and $\mathcal H_U$, then we can proceed inductively to reconstruct all of $\mathcal H$. We will see that these endpoints and sets are completely determined by $(\rho,\alpha)$, from which it will follow that there is a unique $\mathcal H\in\VHC(S_{n-1})$ with $\Phi(\mathcal H)=(\rho,\alpha)$. 

Begin by coloring the elements of $\{1,\ldots,n\}$ so that two elements have the same color if and only if they are in the same block of $\rho$. Make sure to use the color blue to color the elements of the block containing $n$. Let $a$ be the largest element of $\{1,\ldots,n\}$ that is not blue. Because $a+1,\ldots,n-1$ are all blue, we need the points with these heights to see the sky when they look up. This forces us to put $\pi_j=j$ for all $j\in\{a,\ldots,n-1\}$ (otherwise, there would be a hook preventing one of these points from seeing the sky). The northeast endpoint of the top hook of $\mathcal H$ must be the highest point that is not blue in the coloring induced by $\mathcal H$. Our choice of $a$ and the definition of $\Phi$ guarantee that this point has height $a$. Therefore, the northeast endpoint of the top hook of $\mathcal H$ must be $(a,a)$. 

Now, the acyclic orientation $\alpha$ defines a partial order $\preceq$ on the blocks of $\rho$, where we declare that $B\preceq B'$ if and only if there is a directed path from $B$ to $B'$ in $G(\rho)$ or $B=B'$. Let $A$ be the block of $\rho$ containing $a$. One can show that $\mathcal H_S$ must be the union of all of the blocks $D$ satisfying $A\preceq D$. We then know that $\mathcal H_U$ must be $\{1,\ldots,n-1\}\setminus(\mathcal H_S\cup\{a\})$. Observe that the numbers $a+1,\ldots,n-1$ are elements of $\mathcal H_U$; the next-largest entry of $\mathcal H_U$ must be the height of the southwest endpoint of the top hook of $\mathcal H$. More precisely, this southwest endpoint is $(b,c)$, where $b=|\mathcal H_U|-(n-a)$ and $c=\max(\mathcal H_U\setminus \{a+1,\ldots,n-1\})$.  
\end{proof}

Using the notation of Section 2, we now deduce from \eqref{Eq2} that $-k_n(-1)$ is the total number of normalized valid hook configurations on $n-1$ points. In fact, we have the following more general consequence of \eqref{Eq1} and the preceding theorem. 

\begin{corollary}
The $n^\text{th}$ classical cumulant of the free Poisson law with rate $\lambda$ is given by \[k_n(\lambda)=-\sum_{\mathcal H\in\VHC(S_{n-1})}(-\lambda)^{\#\mathcal H+1},\] where $\#\mathcal H$ denotes the number of hooks in $\mathcal H$. 
\end{corollary}

In the previous two sections, we found bijective correspondences among uniquely sorted permutations in $S_{2k+1}$, normalized valid hook configurations on $2k+1$ points with $k$ hooks, and normalized lonely trees on $2k+1$ vertices. We can now finally show that these objects are counted by Lassalle's sequence. Let $\VHC^h(S_{n-1})$ be the set of normalized valid hook configurations on $n-1$ points with $h$ hooks.

\begin{corollary}\label{Cor2}
When $n=2k+2$, the map $\Phi$ from Theorem \ref{Thm2} restricts to a bijection 
\[
	\Phi':\VHC^k(S_{2k+1})\to\widetilde{\mathcal M}^c(2k+2).
\]
In particular, 
\[
	\left|\VHC^k(S_{2k+1})\right|
	=
	A_{k+1}.
\]
\end{corollary}

\begin{proof}
Let $\mathcal H$ be a normalized valid hook configuration on $2k+1$ points with $h$ hooks, and put $\Phi(\mathcal H)=(\rho,\alpha)$. The valid composition $(q_0,\ldots,q_h)$ induced from $\mathcal H$ is a composition of $2k+1-h$ into $h+1$ parts. It follows from the definition of $\Phi$ that $\rho$ has $h+1$ blocks, where the blocks are of sizes $q_0+1,\ldots,q_h+1$ in some order. We find that $\Phi(\mathcal H)\in\widetilde{\mathcal M}^c(2k+2)$ (that is, $\rho$ is a matching) if and only if $q_i=1$ for all $i$. This occurs if and only if $h=k$. This proves the first statement of the corollary. The second statement follows from the first and from \eqref{Eq3}. 
\end{proof}

In the following additional corollary to Theorem \ref{Thm2}, we adopt a notational convention from \cite{Josuat}. Given a sequence $(u_n)_{n\geq 1}$ and a set partition $\rho$, write 
\[
	u_\rho
	=
	\prod_{B\in\rho}u_{|B|}.
\]
For example, if $\rho=\{\{1,4\},\{2,7,8,9\},\{3,5,6\}\}$, then $C_{\rho-1}=C_{2-1}C_{4-1}C_{3-1}=1\cdot 5\cdot 2=10$. Recall the notation from \eqref{Eq1}. 

\begin{corollary}\label{Cor3}
We have 
\[
	\sum_{\rho\in\mathcal P^c(n)}C_{\rho-1}T_{G(\rho)}(1,0)
	=
	(n-1)!.
\]
\end{corollary} 

\begin{proof}
Let $\Phi$ be the bijection from Theorem \ref{Thm2}. Given a normalized valid hook configuration $\mathcal H$ on $n-1$ points, let $\Phi_1(\mathcal H)$ be the set partition which is the first coordinate of $\Phi(\mathcal H)$. In other words, if $\Phi(\mathcal H)=(\rho,\alpha)$, then $\Phi_1(\mathcal H)=\rho$. We know from Theorem \ref{Thm1} and the definition of $\Phi$ that $\displaystyle |s^{-1}(\pi)|=\sum_{\mathcal H\in\VHC(\pi)}C_{\Phi_1(\mathcal H)-1}$ for every $\pi\in S_{n-1}$. Note that the total number of preimages of all permutations in $S_{n-1}$ under $s$ is $(n-1)!$. Invoking Theorem \ref{Thm2}, we find that
\[
	(n-1)!
	=
	\sum_{\pi\in S_{n-1}}|s^{-1}(\pi)|
	=
	\sum_{\pi\in S_{n-1}}\sum_{\mathcal H\in\VHC(\pi)}C_{\Phi_1(\mathcal H)-1}
	=\sum_{\mathcal H\in\VHC(S_{n-1})}C_{\Phi_1(\mathcal H)-1}
	\] \[
	=\sum_{(\rho,\alpha)\in\widetilde{\mathcal P}^c(n)}C_{\rho-1}
	=
	\sum_{\rho\in\mathcal P^c(n)}C_{\rho-1}T_{G(\rho)}(1,0).
\]
We have used the fact, which we mentioned in the paragraph following Theorem \ref{Thm3}, that $T_{G(\rho)}(1,0)$ is the number of acyclic orientations $\alpha$ such that $(\rho,\alpha)\in\widetilde{\mathcal P}^c(n)$. 
\end{proof}

Let $A_{k+1}(\ell)$ be the number of uniquely sorted permutations in $S_{2k+1}$ whose first entry is $\ell$. Corollary \ref{Cor2} provides us with a means for proving the following somewhat surprising result concerning this refinement of the Lassalle numbers. 

\begin{theorem}\label{Thm4}
For each nonnegative integer $k$, the sequence $(A_{k+1}(\ell))_{\ell=1}^{2k+1}$ is symmetric. 
\end{theorem}

\begin{proof}
Let $\Phi_1$ be as in the proof of Corollary~\ref{Cor3}. Let $\widetilde{\mathcal M}_\ell^c(2k+2)$ denote the set of pairs $(\rho,\alpha)\in\widetilde{\mathcal P}^c(2k+2)$ such that $\rho$ is a matching that contains the block $\{\ell,2k+2\}$. If $\mathcal H$ is the unique valid hook configuration of a uniquely sorted permutation $\pi$, then $\{\pi_1,2k+2\}$ is one of the blocks of $\Phi_1(\mathcal H)$. Therefore, $\widetilde{\mathcal M}_\ell^c(2k+2)$ is the image under $\Phi$ of the set of uniquely sorted permutations $\pi\in S_{2k+1}$ such that $\pi_1=\ell$. It now suffices to find a bijection $\widetilde{\mathcal M}_\ell^c(2k+2)\to\widetilde{\mathcal M}_{2k+2-\ell}^c(2k+2)$. 

Suppose $(\rho,\alpha)\in\widetilde{\mathcal M}_\ell^c(2k+2)$, and draw an arch diagram of $\rho$ by connecting two numbers with an arch if and only if they are in the same block. For example, the arch diagram of $\{\{1,4\},\{2,6\},\{3,8\},\{5,7\}\}$ is shown on the left in Figure~\ref{fig-arch-matching}. If we reflect all of the numbers in $\{1,\ldots,2k+1\}$ across the number $k+1$ without breaking any of the arches, we obtain a new matching $\rho'$. More formally, if $\{a,b\}$ is a block of $\rho$ that does not contain $2k+2$, then $\{2k+2-a,2k+2-b\}$ is a block of $\rho'$. Furthermore, $\{2k+2-\ell,2k+2\}$ is a block of $\rho'$. The crossing graphs $G(\rho)$ and $G(\rho')$ are naturally isomorphic, so we can transfer the acyclic orientation $\alpha$ of $G(\rho)$ to an acyclic orientation $\alpha'$ of $G(\rho')$ in the obvious fashion. The map $\widetilde{\mathcal M}_\ell^c(2k+2)\to\widetilde{\mathcal M}_{2k+2-\ell}^c(2k+2)$ given by $(\rho,\alpha)\mapsto(\rho',\alpha')$ is our desired bijection.
\end{proof}

\begin{footnotesize}
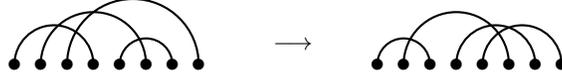
\begin{figure}
\begin{tabular}{ccc}
	\begin{tikzpicture}[scale=0.35]
		\matching{0/3, 1/5, 2/7, 4/6}
	\end{tikzpicture}
	&\quad
	\begin{tikzpicture}[scale=0.35]
		\node at (1,0) {};
		\node at (0,1) {$\longrightarrow$};
	\end{tikzpicture}
	\quad&
	\begin{tikzpicture}[scale=0.35]
		\matching{0/2, 1/5, 3/6, 4/7}
	\end{tikzpicture}
\end{tabular}
\caption{The map $\rho\mapsto\rho'$ described in the proof of Theorem \ref{Thm4}. In this specific example, we have $\rho=\{\{1,4\},\{2,6\},\{3,8\},\{5,7\}\}$ and\\ $\rho'=\{\{1,3\},\{2,6\},\{4,7\},\{5,8\}\}$.}
\label{fig-arch-matching}
\end{figure}
\end{footnotesize}

Numerical evidence suggests that $(A_{k+1}(\ell))_{\ell=1}^{2k+1}$ is log-concave (and therefore unimodal) for each nonnegative integer $k$. We state this as a conjecture in Section 6, where we collect several other suggestions for future work. 

Each uniquely sorted permutation $\pi\in S_{2k+1}$ has a unique valid hook configuration $\mathcal H$, and $\mathcal H$ has a top hook (assuming $k\geq 1$). As discussed above, this hook separates $\mathcal H$ into a sheltered piece and an unsheltered piece. Suppose $(j,\pi_j)$ is the leftmost point in the sheltered piece. We call $\pi_j$ the \emph{eye} of $\pi$. If we let $\Phi'(\alpha)=\rho$, where $\Phi'$ is the map from Corollary \ref{Cor2}, then the eye of $\pi$ is also the entry in the same block as $2k+1$ in $\rho$. We saw in Section 3 that $\pi$ corresponds to a normalized lonely tree $T$. Let $a$ be the left child of the root of $T$. In the notation introduced at the end of Section 3, the eye of $\pi$ is the label of the leftmost cousin of $a$. The following theorem shows an interesting relationship between the first entry and the eye of a uniquely sorted permutation and also provides an alternative method for studying the numbers $A_{k+1}(\ell)$ from Theorem \ref{Thm4}.   

\begin{theorem}\label{Thm5}
If $k\geq 1$, then there are exactly $A_{k+1}(\ell)$ uniquely sorted permutations in $S_{2k+1}$ with eye $\ell-1$. 
\end{theorem}

\begin{proof}
Given a matching $\rho\in\mathcal M^c(2k+2)$ and a number $i\in[2k+2]$, let $\text{par}_\rho(i)$ denote the \emph{partner} of $i$ in $\rho$, which is the unique element in the same block as $i$ in $\rho$. Let $\rho^*$ be the matching of $\{0,\ldots,2k+1\}$ obtained from $\rho$ by replacing the number $2k+2$ with $0$. Let $\rho^{**}$ be the matching obtained by reflecting $\rho^*$ about the number $\dfrac{2k+1}{2}$. In other words, if $\rho^*=\{\{a_1,b_1\},\ldots,\{a_{k+1},b_{k+1}\}\}$, then $\rho^{**}=\{\{2k+1-a_1,2k+1-b_1\},\ldots,\{2k+1-a_{k+1},2k+1-b_{k+1}\}\}$. Finally, let $\widetilde\rho$ be the matching in $\mathcal M(2k+2)$ obtained by replacing the number $0$ in $\rho^{**}$ with $2k+2$. It is straightforward to check that the crossing graphs $G(\rho)$ and $G(\widetilde\rho)$ are isomorphic. Therefore, every acyclic orientation $\alpha$ of $G(\rho)$ corresponds in the obvious way to an acyclic orientation $\widetilde\alpha$ of $G(\widetilde\rho)$. Since $\rho$ is connected, $\text{par}_\rho(2k+2)\neq 2k+1$. This implies that $\text{par}_{\rho^*}(0)\neq 2k+1$, so $\text{par}_{\rho^{**}}(2k+1)\neq 0$. Thus, \[\text{par}_\rho(2k+2)=\text{par}_{\rho^*}(0)=2k+1-\text{par}_{\rho^{**}}(2k+1)=2k+1-\text{par}_{\widetilde\rho}(2k+1).\] As a consequence, we obtain a bijection \[\{(\rho,\alpha)\in\widetilde{\mathcal M}^c(2k+2):\text{par}_{\rho}(2k+2)=2k+2-\ell\}\to\{(\widetilde\rho,\widetilde\alpha)\in\widetilde{\mathcal M}^c(2k+2):\text{par}_{\widetilde\rho}(2k+1)=\ell-1\}.\] 

Now consider the bijection $\Phi'$ from Corollary \ref{Cor2} (which is a restriction of the map $\Phi$). The set \[(\Phi')^{-1}(\{(\rho,\alpha)\in\widetilde{\mathcal M}^c(2k+2):\text{par}_{\rho}(2k+2)=2k+2-\ell\})\] is the set of valid hook configurations in $\VHC^k(S_{2k+1})$ in which the leftmost point in the plot has height $2k+2-\ell$. This set is naturally in bijection (by just taking the underlying permutation of each valid hook configuration) with the set of uniquely sorted permutations $\pi\in S_{2k+1}$ with $\pi_1=2k+2-\ell$. By definition, the size of this set is $A_{k+1}(2k+2-\ell)$. Therefore, it follows from the above bijection that \[|\{(\widetilde\rho,\widetilde\alpha)\in\widetilde{\mathcal M}^c(2k+2):\text{par}_{\widetilde\rho}(2k+1)=\ell-1\}|=A_{k+1}(2k+2-\ell).\] The set \[(\Phi')^{-1}(\{(\widetilde\rho,\widetilde\alpha)\in\widetilde{\mathcal M}^c(2k+2):\text{par}_{\widetilde\rho}(2k+1)=\ell-1\})\] is naturally in bijection with the set of uniquely sorted permutations in $S_{2k+1}$ with eye $\ell-1$. Hence, the number of uniquely sorted permutations in $S_{2k+1}$ with eye $\ell-1$ is $A_{k+1}(2k+2-\ell)$. The desired result now follows from Theorem~\ref{Thm4}, which tells us that $A_{k+1}(2k+2-\ell)=A_{k+1}(\ell)$. 
\end{proof}

\section{New Recurrences}

We know from \eqref{Eq2} and Theorem \ref{Thm2} that the total number of normalized valid hook configurations on $n$ points is $-k_{n+1}(-1)$. We know from \eqref{Eq3} and Corollary \ref{Cor2} that the total number of normalized valid hook configurations on $2k+1$ points with $k$ hooks is $A_{k+1}$. In this section, we study the combinatorial properties of valid hook configurations in order to derive new recurrence relations for these numbers. These recurrences keep track of a permutation statistic that we call the \emph{tail length}. In what follows, the \emph{normalization} of a permutation $\pi=\pi_1\cdots\pi_n$ is the unique permutation in $S_n$ that is order isomorphic to $\pi$. For example, the normalization of $26589$ is $13245$.  

\begin{definition}
The \emph{tail length} of a permutation $\pi=\pi_1\cdots\pi_n\in S_n$, denoted $\text{tl}(\pi)$, is the smallest nonnegative integer $i$ such that $\pi_{n-i}\neq n-i$. The tail length of an arbitrary permutation is the tail length of its normalization. 
\end{definition}

For example, the permutation $31524678$ has tail length $3$, while the permutation $26589$ has tail length $2$. For an indication of the relevance of this statistic for our purposes, observe that a sorted permutation (equivalently, a permutation that has a valid hook configuration) must have a positive tail length. Heuristically, we should expect the fertility of a permutation in $S_n$ with a large tail length to be larger than the fertility of a permutation in $S_n$ with a small tail length. 

Let 
\[
	D_m(n)
	=
	\sum_{\substack{\pi\in S_n\\ \text{tl}(\pi)=m}}|\VHC(\pi)|
\] 
be the total number of valid hook configurations whose underlying permutations are elements of $S_n$ with tail length $m$. Let 
\[
	D_{\geq m}(n)
	=
	\sum_{\ell=m}^{n} D_{\ell}(n).
\] 
In particular, $D_{\geq 0}(n)=-k_{n+1}(-1)$ is the total number of normalized valid hook configurations on $n$ points.  

\begin{theorem}\label{Thm6}
The numbers $D_m(n)$ and $D_{\geq m}(n)$ defined above satisfy the recurrence
\[
	D_m(n)
	=
	\sum_{j=1}^m\sum_{i=1}^{n-m-1}{n-m-1\choose i-1}D_{\geq j}(i+j-1)D_{\geq m-j}(n-j-i)
\] 
for $0\leq m<n$. The initial conditions are given by $D_n(n)=1$. 
\end{theorem}

\begin{proof}
The initial condition $D_n(n)=1$ is the statement that the identity permutation has a unique valid hook configuration (the one with no hooks). The recurrence is obvious when $0=m<n$ since a permutation in $S_n$ with tail length $0$ has no valid hook configurations. 

Now suppose $0<m<n$. To produce a valid hook configuration of a permutation $\pi=\pi_1\cdots\pi_n\in S_n$ with $\text{tl}(\pi)=m$, begin by choosing the index $i\in\{1,\ldots,n-m-1\}$ such that $\pi_i=n-m$. Note that $i$ must be a descent of this permutation. This implies that there must be a hook $H$ with southwest endpoint $(i,n-m)$. The northeast endpoint of this hook is of the form $(n-j,n-j)$ for some $j\in\{0,\ldots,m-1\}$. There are ${n-m-1\choose i-1}$ choices for the entries in the set $\{\pi_1,\ldots,\pi_{i-1}\}$. Note that $\pi_{i+1}\cdots\pi_{n-j}$ will be a (not necessarily normalized) permutation of length $n-j-i$ with tail length at least $m-j$. Choosing the part of the valid hook configuration that lies below $H$ amounts to choosing $\pi_{i+1}\cdots\pi_{n-j}$ and choosing a valid hook configuration on this permutation. There are $D_{\geq m-j}(n-j-i)$ ways to do this. Similarly, there are $D_{\geq j}(i+j-1)$ ways to choose the hooks on the points that are not $(n-j,n-j)$ and do not lie below $H$. 
\end{proof}

\begin{footnotesize}
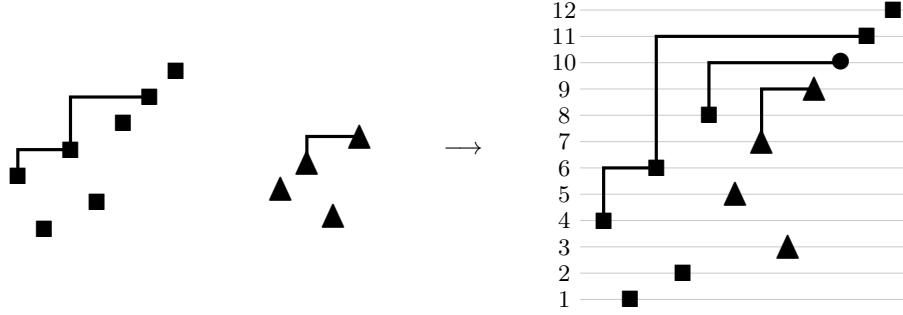
\begin{figure}
\begin{tabular}{ccccccc}
	\begin{tikzpicture}[scale=0.35]
	
		\foreach \val [count=\idx] in {3,1,4,2,5,6,7} {
			\abssquare{(\idx,\val)}
		}
		\hook{1}{3}{3}{4}
		\hook{3}{4}{6}{6}
		
		\node at (0,-2) {};
	\end{tikzpicture}
	&\quad&
	\begin{tikzpicture}[scale=0.35]
	
		\foreach \val [count=\idx] in {2,3,1,4} {
			\abstriangle{(\idx,\val)}
		}
		\hook{2}{3}{4}{4}
		
		\node at (0,-2.5) {};
	
	\end{tikzpicture}
	&\quad&
	\begin{tikzpicture}[scale=0.35]
	
		\node at (0,0) {};
		\node at (0,6) {$\longrightarrow$};
	
	\end{tikzpicture}
	&\quad&
	\begin{tikzpicture}[scale=0.35]
	
		\foreach \val in {1,...,12} {
			\node at (-0.5, \val) {$\val$};
			\draw [lightgray] (0.1,\val) -- (12.5,\val);
		}
		% {4,1,6,2,8,5,7,3,9,10,11,12}
		\foreach \val [count=\idx] in {4,1,6,2,8} {
			\abssquare{(\idx,\val)}
		}
		
		\foreach \val [count=\idx] in {5,7,3,9} {
			\abstriangle{(\idx+5,\val)}
		}
		
		\absdot{(10,10)}
		\abssquare{(11,11)}
		\abssquare{(12,12)}
		
		\hook{1}{4}{3}{6}
		\hook{3}{6}{11}{11}
		\hook{5}{8}{10}{10}
		\hook{7}{7}{9}{9}	
	
	\end{tikzpicture}
\end{tabular}
\caption{Two small valid hook configurations combine into a larger one as described in the proof of Theorem \ref{Thm6}.}
\label{fig-recurrence}
\end{figure}
\end{footnotesize}

\begin{example}\label{Exam1}
Figure~\ref{fig-recurrence} shows the construction of a valid hook configuration of a permutation $\pi\in S_{12}$ with tail length $4$. Here, we have chosen $i=5$, $j=3$, and $\{\pi_1,\pi_2,\pi_3,\pi_4\}=\{1,2,4,6\}$. We have also chosen two small valid hook configurations, which combine to form the large valid hook configuration on the right. The points coming from the first small valid hook configuration are represented as squares while the points coming from the second one are represented as triangles. Note that the point $(10,10)$ is represented by a disc because it does not come from either of these small valid hook configurations (this is because it is the northeast endpoint of the hook $H$).   
\end{example}

Let $E_m(n)$ be the number of uniquely sorted permutations in $S_n$ with tail length $m$. Of course, $E_m(n)=0$ when $n$ is even since there are no uniquely sorted permutations of even length. Let 
\[
	E_{\geq m}(n)
	=
	\sum_{\ell=m}^nE_\ell(n).
\] 
In particular, $E_{\geq 0}(2k+1)=A_{k+1}$ is the total number of uniquely sorted permutations in $S_{2k+1}$. 

\begin{theorem}\label{Thm7}
The numbers $E_m(n)$ and $E_{\geq m}(n)$ defined above satisfy the recurrence
\[
	E_m(n)
	=
	\sum_{j=1}^m\sum_{i=1}^{n-m-1}{n-m-1\choose i-1}E_{\geq j}(i+j-1)E_{\geq m-j}(n-j-i)
\] 
for $0\leq m<n$. The initial conditions are given by $E_1(1)=1$ and $E_n(n)=0$ for $n\neq 1$. 
\end{theorem}

\begin{proof}
Suppose $\mathcal H$ is a valid hook configuration of a permutation $\pi$. It follows from the discussion preceding Proposition~\ref{Prop1} that $\pi$ is uniquely sorted if and only if the coloring of the plot of $\pi$ induced by $\mathcal H$ does not give any two distinct points the same color. The proof of Theorem~\ref{Thm7} is now virtually identical to that of Theorem~\ref{Thm6}. Specifically, we start to construct the valid hook configuration of a uniquely sorted permutation $\pi\in S_n$ with $\text{tl}(\pi)=m$ by first choosing the index $i$ such that $\pi_i=n-m$. As before, $(i,n-m)$ must be the southwest endpoint of a hook $H$. We choose $j$ such that $(n-j,n-j)$ is the northeast endpoint of $H$. We then choose the set $\{\pi_1,\ldots,\pi_{i-1}\}$ in ${n-m-1\choose i-1}$ ways. Finally, we choose the part of the valid hook configuration lying below $H$ in $E_{\geq m-j}(n-j-i)$ ways and choose the part not lying below $H$ in $E_{\geq j}(i+j-1)$ ways.        
\end{proof}

\section{Future Work}
Through different ways of phrasing the main results of our paper, we obtain several possible avenues for potential generalizations. For example, it would be interesting to enumerate doubly sorted permutations, which  are simply permutations with fertility $2$. Arguing as in Section 2, one can show that there are no doubly sorted permutations of odd length. Letting $B_m$ denote the number of doubly sorted permutations of length $2m$, we have $B_1=1$, $B_2=3$, $B_3=31$, and $B_4=1186$. This sequence appears to be new. 

A permutation in $S_n$ is uniquely sorted if and only if it is sorted and has $\frac{n-1}{2}$ descents. From this point of view, it would be interesting to count sorted permutations in $S_n$ with exactly $k$ descents, where $k$ could be a function of $n$. For example, we could ask how many sorted permutations in $S_n$ have exactly $\frac{n-2}{2}$ descents. It might also be interesting to enumerate uniquely sorted permutations according to certain statistics, such as the number of inversions, the major index, or the number of peaks. 

Of course, uniquely sorted permutations in $S_n$ are in bijection with permutations $\pi\in S_n$ such that $s(\pi)$ is uniquely sorted. As mentioned in the previous paragraph, $s(\pi)$ is uniquely sorted if and only if $s(\pi)$ has exactly $\frac{n-1}{2}$ descents. This leads us to ask for the total number of permutations $\pi\in S_n$ such that $s(\pi)$ has exactly $k$ descents. Again, $k$ could be a function of $n$ here. 

We have seen that every uniquely sorted permutation has exactly one valid hook configuration. It could be interesting to count the total number of permutations in $S_n$ that have exactly one valid hook configuration. 

The current author and Kravitz \cite{DefantKravitz} have formulated two extensions of the stack-sorting map defined on words. It might be fruitful to consider the appropriate notions of ``uniquely sorted words.''

In Theorem \ref{Thm4}, we saw that the sequence $(A_{k+1}(\ell))_{\ell=1}^{2k+1}$ is symmetric for each nonnegative integer $k$. Recall that a sequence $a_1,\ldots,a_m$ is called \emph{unimodal} if there exists $j\in\{1,\ldots,m\}$ such that $a_1\leq\cdots\leq a_{j-1}\leq a_j\geq a_{j+1}\geq\cdots\geq a_m$ and is called \emph{log-concave} if $a_j^2\geq a_{j-1}a_{j+1}$ for all $j\in\{2,\ldots,m-1\}$ \cite{Branden}. It is well-known that a log-concave sequence of nonnegative real numbers is unimodal.

\begin{conjecture}\label{Conj1}
For each nonnegative integer $k$, the sequence $(A_{k+1}(\ell))_{\ell=1}^{2k+1}$ is log-concave. 
\end{conjecture}

Even if Conjecture \ref{Conj1} is too difficult to prove, it would still be of great interest to prove the weaker claim that these sequences are unimodal. We have verified Conjecture \ref{Conj1} for $0\leq k\leq 5$. 

When defining the bijection $\Phi$ that appears in Theorem \ref{Thm2}, we described how to obtain a partition $\rho=\{B_1,\ldots,B_{k+1}\}$ of $\{1,\ldots,n\}$ from a valid hook configuration $\mathcal H$ of a permutation $\pi=\pi_1\cdots\pi_{n-1}\in S_{n-1}$. This is done by first coloring the points in the plot of $\pi$ and then coloring the number $\pi_i$ the same color as the point $(i,\pi_i)$. After coloring $n$ blue, we obtain a partition of $\{1,\ldots,n\}$ into color classes. In the proof of Theorem \ref{Thm2}, we showed that this partition is connected (that is, its crossing graph is connected). We can obtain another set partition $\eta=\{\widehat{B}_1,\ldots,\widehat{B}_{k+1}\}$ from $\mathcal H$. To do this, color the points $(i,\pi_i)$ as before, but this time, color a number $i$ the same color as $(i,\pi_i)$. This will partition $\{1,\ldots,n-1\}$ into color classes. It follows from the rules defining valid hook configurations that $\eta$ is a noncrossing partition (that is, its crossing graph has no edges). It would be interesting to investigate possible connections between the partitions $\rho$ and $\eta$ that are obtained from the same valid hook configuration $\mathcal H$. We could also study the noncrossing partitions arising in this way in their own right. Noncrossing partitions are fundamental objects in the combinatorics of free probability theory, so it would be interesting to see if the noncrossing partitions obtained from valid hook configurations in this manner have some deeper significance.    

Let us remark that the first author has now extended the investigations initiated in this article by considering uniquely sorted permutations and valid hook configurations that avoid various patterns \cite{DefantCatalan, DefantMotzkin}. Pattern-avoiding uniquely sorted permutations were studied further by Mularczyk \cite{Hanna}, and pattern-avoiding valid hook configurations were studied further by Sankar \cite{Maya}. There is still much to be done in both of these lines of work. In particular, see the end of Sankar's paper for some remarkable conjectures about $312$-avoiding ``reduced'' valid hook configurations.  

\vspace{-.2cm}

\section{Acknowledgments}
This work is partially supported by NSF-DMS grants 1603823 and 1604458. The first author was supported by a Fannie and John Hertz Foundation Fellowship and an NSF Graduate Research Fellowship (grant number DGE-1656466). We thank Levent Alpoge and Zachary Hilliard for helpful discussions. We also thank the anonymous referee for helpful comments.

\end{document}